\documentclass[11pt,reqno]{amsart}
\usepackage[margin=1.15in]{geometry}
\usepackage{amsmath,amssymb,amsthm,mathtools,booktabs,microtype}
\usepackage{xcolor}
\usepackage{graphicx}
\usepackage{tabularx}
\usepackage[colorlinks,linkcolor=blue!55!black,citecolor=blue!55!black,
            urlcolor=blue!55!black]{hyperref}

\theoremstyle{plain}
\newtheorem{theorem}{Theorem}[section]
\newtheorem{lemma}[theorem]{Lemma}
\newtheorem{proposition}[theorem]{Proposition}
\newtheorem{corollary}[theorem]{Corollary}

\theoremstyle{definition}
\newtheorem{definition}[theorem]{Definition}
\newtheorem{remark}[theorem]{Remark}

\DeclareMathOperator{\Real}{Re}
\DeclareMathOperator{\Imag}{Im}
\DeclareMathOperator{\capa}{cap}
\DeclareMathOperator{\dist}{dist}

\DeclareMathOperator{\artanh}{artanh}
\newcommand{\D}{\mathbb D}
\newcommand{\C}{\mathbb C}
\newcommand{\Chat}{\widehat{\mathbb C}}
\newcommand{\U}{\mathcal U}
\newcommand{\Sclass}{\mathcal S}
\newcommand{\B}{\mathcal B}
\DeclarePairedDelimiter{\abs}{\lvert}{\rvert}
\DeclarePairedDelimiter{\norm}{\lVert}{\rVert}

\begin{document}
\title[Bounds for the univalent Bloch--Landau constant]
{Bounds for the univalent Bloch--Landau constant}
\author{Frank Wikstr\"om}
\address{Centre for Mathematical Sciences, Lund University,
Box 118, SE-221 00 Lund, Sweden}
\email{frank.wikstrom@math.lth.se}

\begin{abstract}
The univalent Bloch--Landau constant $\U$ is the largest number such that, for
every univalent function $f$ on the unit disk, $f(\D)$ contains a disk of radius
$\U\abs{f'(0)}$. We prove $0.619\le\U<0.6562681$. The previous lower bound,
$0.57088586$, came from growth estimates, which control disks centered at
$f(0)$, and we show that for Landau-normalized functions such disks cannot give
a bound above $0.6$. The main new idea is to move the center of the disk. After Landau's
normalization, a compression of the Grunsky matrix constrains every omitted
value through the same few coefficients, so a center chosen from these
coefficients can be kept away from all omitted values at once. A short analytic
proof based on this idea gives $\U\ge0.592$, and a three-row compression,
together with pointwise consequences of Landau's normalization, improves the
lower bound to $\U\ge0.619$ through an interval certificate. For the upper bound,
we construct an explicit domain, bounded by Goodman's first two slit
generations, thirty analytic arcs and a circle. We certify $\U<0.6562681$ with a
capacity estimate for almost critical arcs, improving the best published bound
$0.6563937$.
\end{abstract}
\maketitle

\section{Introduction}
For a univalent function $f$ on the unit disk $\D$, let $R_f$ denote the
inradius of $f(\D)$, i.e.\ the supremum of the radii of disks contained in its
image. The \emph{univalent Bloch--Landau constant} is
\[
  \U=\inf \left\{\frac{R_f}{\abs{f'(0)}}:f\text{ univalent in }\D\right\}.
\]
Landau~\cite{Landau} introduced this constant and proved $\U>0.566$ in 1929.
Successive refinements by Reich~\cite{Reich}, Jenkins~\cite{Jenkins61},
Toppila~\cite{Toppila}, Zhang~\cite{Zhang}, Jenkins~\cite{Jenkins98},
Xiong~\cite{Xiong}, and Skinner~\cite{Skinner} raised this to
$0.57088586$~\cite[Theorem 3.2]{Skinner} in 2009.
Upper bounds come from explicit domains. Robinson~\cite{Robinson} showed
$\U<0.658$ in 1935, and constructions by Goodman~\cite{Goodman}, Beller and
Hummel~\cite{BellerHummel}, and Carroll and
Ortega-Cerd\`a~\cite{CarrollOrtegaCerda} lowered this to $0.6563937$ in 2009.

The distinction between a centered disk and a disk with freely chosen
center is central to our approach. For a normalized schlicht function,
put $d_f=\dist\bigl(0, \C\setminus f(\D)\bigr)$, and
\[
  R_f=\sup_{c\in f(\D)}\dist\bigl(c, \C\setminus f(\D)\bigr).
\]
A lower growth estimate for $\abs{f(z)}$ gives a lower bound for $d_f$, and
hence for $R_f$. The line of argument culminating in Skinner's result
uses Landau's normalization, an explicit lower growth estimate obtained
from the Fekete--Szeg\H{o} inequality, and an implicit improvement of that
estimate; see~\cite[Sections 2--3]{Skinner}. Its final estimate controls
the distance to the nearest omitted value measured from the origin.
The inradius, however, allows the center to move. A small value of $d_f$
need not force an equally small value of $R_f$: the relevant issue is
whether omitted values in other directions prevent a translation of the
disk.

We keep Landau's normalization and use only low-order coefficient information,
but retain the angular dependence of the omitted-value constraints while
choosing the center. A $2\times2$ compression of the Grunsky operator supplies
both the supporting inequality for an analytic proof and the complex matrix
constraint for a computer-assisted proof, and a $3\times3$ compression refines
the latter. The lower bounds are as follows.

\begin{theorem}[Analytic bound]\label{thm:hand}
The univalent Bloch--Landau constant satisfies
\[
 \U \ge 0.592.
\]
\end{theorem}

The proof treats one small-coefficient case and three further ranges of the
cubic coefficient, each with its own supporting parameter. Its scalar
estimates are a small number of explicit inequalities, collected in
Appendix~\ref{app:scalar}. Each holds with a margin of at least $10^{-4}$, and
no computer search or adaptive subdivision is needed. A Lean~4 formalization
proves Theorem~\ref{thm:hand} from two classical coefficient inequalities for
univalent functions, a $2\times2$ Grunsky inequality and the Fekete--Szeg\H{o}
inequality, which enter as explicit hypotheses. It covers Landau's
normalization, the bounded-image transform and all scalar
estimates~\cite{WikArchive}.

The stronger bounds require computer assistance.
Let $\Sclass$ be the class of univalent functions
$f(z)=z+\sum_{n \ge 2}a_nz^n$ on $\D$.
\begin{theorem}[Computer-assisted bounds]\label{thm:computer}
\leavevmode
\begin{enumerate}
\item[(a)] Every $f\in\Sclass$ with $a_2=0$ and $\abs{a_3} \le 1/3$ satisfies
 \[
  R_f \ge \frac{306121}{500000}=0.612242.
 \]
\item[(b)] Every $f\in\B$ satisfies $R_f\ge 619/1000$.
\end{enumerate}
In particular, $\U \ge 0.619$.
\end{theorem}

Here $\B$ is the class of Landau-normalized functions of
Section~\ref{sec:normalization}; it is contained in the class of part~(a).
Theorem~\ref{thm:computer} raises the lower bound by more than $0.048$,
compared with a total gain of less than $0.005$ between 1929 and 2009, and
removes more than half of the gap to the upper bound $0.6562681$ of
Theorem~\ref{thm:upper-intro} below. Part~(a) holds for a larger class and
has a much shorter proof, which already contains the main idea.

To explain the mechanism, Landau's normalization and a rotation allow
us to write
\[
 f(z)=z-sz^3+az^4+\cdots,\qquad 0 \le s \le 1/3,
\]
where $a$ may be complex. For each omitted value $w$, the transform $f/(1-f/w)$
is again normalized and univalent. Applying the Grunsky inequalities to its odd
square-root transform gives the matrix $M_{s,a}(w)$ in
\eqref{eq:omitted-matrix}, with $\norm{M_{s,a}(w)} \le 1$. Consequently, every
omitted value is contained in the set
\[
 K_{s,a}=\{w \ne 0: \norm{M_{s,a}(w)} \le 1\}.
\]
The set $K_{s,a}$ is an outer relaxation of the omitted set, with the
coefficients held fixed, and it need not itself be the omitted set of a schlicht
function. We seek a disk disjoint from $K_{s,a}$, by choosing its center as a
function of $s$ and $a$, see Figure~\ref{fig:kset}. This preserves information
that is lost if the coefficient terms are separately replaced by their
worst-case absolute values before the center is chosen.

There are two complementary uses of this information. In the analytic proof, a
real supporting form of the matrix inequality keeps a second angular harmonic
from~$s$ and a third harmonic from~$a$. A suitable translation contributes a
first harmonic; subtracting it from the quartic-coefficient term leaves a factor
$\sin^2\theta$, which can be controlled together with the second harmonic. This
explains the cancellation behind the larger disk, see Figure~\ref{fig:analytic}.
The bounded-image transform and the classical growth estimate are still used,
but the angular terms are kept until this cancellation has taken place. In the
computer-assisted proof, complex quadratic forms take into account more of the
matrix constraint, and a polynomial choice of center gives part~(a) of
Theorem~\ref{thm:computer}. Its radius is within $10^{-6}$ of a limit for the
$2\times2$ matrix, which occurs at $s=1/3$, $a=0$
(Section~\ref{sec:obstruction}). This limit is local: with centers chosen
separately on small parameter boxes, the same matrix gives the radius $0.619$
outside a small part of the parameter set. There the proof of part~(b) adds a
third Grunsky row, which involves $a_5$ and $a_6$, and lets the center depend
on these coefficients as well. Part~(a) requires only $a_2=0$ and $\abs{a_3} \le 1/3$, rather than the
full Bloch derivative inequality supplied by Landau's normalization. Part~(b)
uses that inequality, but only through elementary bounds for $f'$ at a few
points near the origin.

In summary, the larger gain comes from combining joint coefficient constraints
with the freedom to move the disk center. For
Theorem~\ref{thm:computer} this freedom is necessary: Proposition~\ref{prop:centered}
supplies a function in~$\B$ whose image omits a point at distance less than~$0.6$
from~$f(0)$, so no estimate for centered disks that holds on all of~$\B$ can
exceed~$0.6$.

In both parts of Theorem~\ref{thm:computer}, the statement requiring
computation is a finite family of explicit inequalities, verified exhaustively
with Arb interval arithmetic. Floating-point computations propose test vectors,
but every condition used to accept a parameter box is checked rigorously. The
role of computation is thus to verify finitely many explicit inequalities after
the analytic reduction, with no discretization of arbitrary univalent maps or
numerical integration of a L\"owner equation. For part~(a), an explicit
polynomial center and an elementary radial monotonicity criterion leave scalar
inequalities on four real parameters (Lemma~\ref{lem:finite-certificate}).
Their verification takes about $42$~seconds on the author's laptop, and an
independently written second verifier confirms the result. For part~(b), the
inequalities also involve $a_5$ and $a_6$ and come from dual solutions of
semidefinite programs. The stored certificate is checked in under ten minutes
on a desktop computer.

In the other direction, we give a slight improvement of the upper bound, by
keeping Goodman's first two generations of slits, and then adding five arcs in
each sector of opening $\pi/3$, designed to be harmonically symmetric, and
finally truncating by a circle.

\begin{theorem}\label{thm:upper-intro}
There is a univalent function $f$ on $\D$, mapping onto an explicit domain,
with $R_f/\abs{f'(0)}<0.656268014$. Consequently $\U<0.6562681$.
\end{theorem}

This improves the bound $0.6563937$ of Carroll and
Ortega-Cerd\`a~\cite{CarrollOrtegaCerda}. It is also below the value
$0.656319277$ that Ortega-Cerd\`a and
Pridhnani~\cite[Section 5.3]{OrtegaCerdaPridhnani} computed numerically,
without a rigorous error analysis. The proof combines a capacity estimate
for almost critical arcs with an interval-arithmetic certificate for the
inradius, see Theorem~\ref{thm:upper}. After this work was completed, van
Staden~\cite{vanStaden} posted a preprint claiming $\U<0.65628589$ by a
different computer-assisted method, using a disk with twenty-four slits and a
comparison function controlled in ball arithmetic. The bound of
Theorem~\ref{thm:upper-intro} is smaller, as is the simpler three-arc bound
$0.6562745$ of Section~\ref{sec:upper}.

Sections~\ref{sec:normalization} and~\ref{sec:grunsky} give the common
analytic preliminaries. Section~\ref{sec:hand} proves
Theorem~\ref{thm:hand}, Sections~\ref{sec:matrix} and~\ref{sec:certificate}
prove part~(a) of Theorem~\ref{thm:computer}, Section~\ref{sec:hybrid} proves
part~(b), and Section~\ref{sec:upper} proves
Theorem~\ref{thm:upper-intro}.

The results rest on different kinds of verification, summarized in
Table~\ref{tab:verification}. All verifiers, data and the Lean development are
available in the archive~\cite{WikArchive}.

\begin{table}[tbp]
\caption{The kinds of verification behind the results.}
\label{tab:verification}
\centering
\small
\begin{tabularx}{\textwidth}{lX}
\toprule
Result & Verification\\
\midrule
Theorem~\ref{thm:hand} & Proof in Section~\ref{sec:hand}; exact rational
checks of the scalar estimates (Appendix~\ref{app:scalar}); Lean~4
formalization from the Grunsky and Fekete--Szeg\H{o} inequalities\\
Theorem~\ref{thm:computer}(a) & Exhaustive interval certificate
(Lemma~\ref{lem:finite-certificate}) in Arb, confirmed by an independently
written verifier\\
Theorem~\ref{thm:computer}(b) & Interval certificates in Arb for four kinds of
parameter boxes (Lemma~\ref{hyb:lem:finite}); stored certificates replayed
without floating-point search; exact check of the covering\\
Proposition~\ref{prop:centered} & Exact Bernstein-coefficient and rational
Riemann-sum checks (Appendix~\ref{app:centered})\\
Theorem~\ref{thm:upper-intro} & Exact curve data, Taylor-model capacity
estimates and a geometric covering certificate, in Arb\\
\bottomrule
\end{tabularx}
\end{table}

\subsection*{Tool disclosure}
OpenAI (GPT-6 Astra) as well as Anthropic Claude (Opus~5.5 and Fable~5.1), were
used as interactive research assistants for mathematical brainstorming and proof
auditing.

The models assisted with numerical exploration as well as with developing the
interval-arithmetic verification scripts. The Lean certificate for the analytic
bound was generated using Opus~5.5. The author has reviewed all generated
material, independently checked the arguments and computations, and takes full
responsibility for the content of the paper.

\section{Landau normalization}\label{sec:normalization}

For convenience of the reader, let us recall the Landau normalization. Put
\[
 \B=\{f\in\Sclass:(1-\abs{z}^2)\abs{f'(z)} \le 1\text{ for }z\in\D\}.
\]
\begin{lemma}\label{lem:normalization}
We have
\[
 \U=\inf_{f\in\B}R_f.
\]
Every $f\in\B$ satisfies $a_2=0$, $\abs{a_3} \le 1/3$, and
$\abs{f(z)} \le \artanh\abs{z}$.
\end{lemma}
\begin{proof}
Fix a univalent $f$ and $0<r<1$, and put $g(z)=f(rz)$. The function
$(1-\abs{z}^2)\abs{g'(z)}$ attains its positive maximum $M$ at a point $a\in\D$.
Let $\phi_a(z)=(a+z)/(1+\bar a z)$. The normalized map
\[
 G(z)=\frac{g(\phi_a(z))-g(a)}{(1-\abs{a}^2)g'(a)}
\]
belongs to $\B$, by the identity
$(1-\abs{z}^2)\abs{\phi_a'(z)}=1-\abs{\phi_a(z)}^2$. Also $M \ge
\abs{g'(0)}=r\abs{f'(0)}$ and $R_g=MR_G \le R_f$. Letting $r\nearrow 1$ proves
the nontrivial direction of the asserted infimum identity; the other follows
from $\B\subset\Sclass$. This is Landau's normalization, see also \cite[Section
1]{Skinner}.

Expanding the Bloch inequality at zero in every direction first gives
$a_2=0$ and then $3\abs{a_3} \le 1$. Integrating
$\abs{f'(te^{i\theta})} \le (1-t^2)^{-1}$ along a radius gives the last claim.
\end{proof}

After a rotation $f(z)\mapsto e^{-i\theta}f(e^{i\theta}z)$, we may
therefore write
\begin{equation}\label{eq:f}
 f(z)=z-sz^3+az^4+\cdots,\qquad\text{where $0 \le s \le 1/3$.}
\end{equation}
The quartic coefficient $a$ is in general complex. The analytic proof
uses $f\in\B$, including its growth bound. The matrix proof of
Theorem~\ref{thm:computer}(a) uses only $f\in\Sclass$ and the bounds on $a_2$
and $a_3$. The proof of part~(b) uses $f\in\B$ through pointwise bounds for
$f'$ near the origin.

Both proofs end by showing that a suitable open disk contains no omitted
value, that is, no point of $\C\setminus f(\D)$. Such a disk lies in
$f(\D)$, so $R_f$ is at least its radius.

\subsection{A limit for disks centered at the origin}
Every lower bound for $d_f=\dist(0,\C\setminus f(\D))$ that holds for all
$f\in\B$ is also a lower bound for $\U$; Skinner's bound $0.57088586$ is of
this kind. The following example shows that such centered bounds cannot
exceed $0.6$. Hence Theorem~\ref{thm:computer}, and
any bound above $0.6$, requires a disk that is not centered at $f(0)$.

\begin{proposition}\label{prop:centered}
Let $a=2/5$, $\eta=999/1000$, and
\[
 H(z)=\frac{(1-z^2)(1+az+z^2)}{(1+z^2)(1+az+az^3+z^4)},\qquad
 f_0(z)=\int_0^zH(\eta\zeta)\,d\zeta .
\]
Then $f_0\in\B$, $f_0$ is univalent in a neighborhood of $\overline\D$, and
$0<f_0(1)<0.599847$. In particular $d_{f_0} \le f_0(1)<0.6$.
\end{proposition}

The proof is given in Appendix~\ref{app:centered}. The example only motivates
the translated disks; it is not used in the proofs of the lower bounds.

The domain $f_0(\D)$ is bounded, and its boundary has four long, thin
incursions resembling horizontal slits: one from the right, ending at
$f_0(1)$, and three from the left; see Figure~\ref{fig:centered}. The
center rule~\eqref{mat:eq:center} of Section~\ref{sec:matrix}, applied to the
coefficients of~$f_0$, moves the disk away from the omitted
point~$f_0(1)$. Together with Skinner's bound, Proposition~\ref{prop:centered}
shows that the largest radius obtainable from centered disks,
$\inf_{f\in\B}d_f$, lies in $[0.5708,0.6)$. We do not know whether
arguments based on centered disks could reach~$0.592$.

\begin{figure}[ht]
\centering
\includegraphics[width=\textwidth]{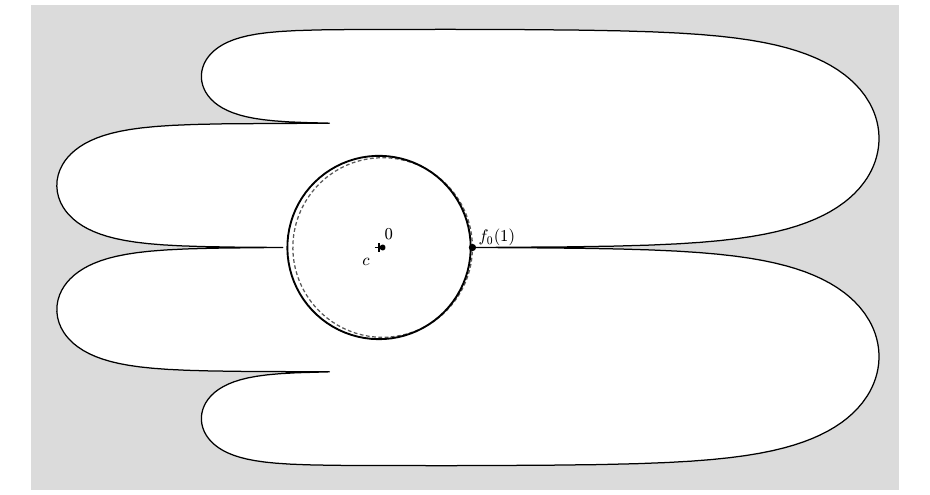}
\caption{The domain $f_0(\D)$ of Proposition~\ref{prop:centered} (white).
The disk centered at $f_0(0)=0$ (dashed) is stopped by the omitted point
$f_0(1)=0.5998$. The disk $D(c,R)$ of Theorem~\ref{mat:thm:disk}, with
$R=0.612242$ and the center $c=-0.0242$ of \eqref{mat:eq:center} computed
from the coefficients of $f_0$, lies in $f_0(\D)$ (solid); the distance from
its center to the boundary is $0.624$. The values $0.5998$, $0.0242$ and
$0.624$ are rounded floating-point values.}
\label{fig:centered}
\end{figure}
\section{A Grunsky matrix and its supporting forms}\label{sec:grunsky}

We use the classical Grunsky inequalities; see \cite{Pommerenke} and
\cite[Section 1]{OT} for the odd square-root transform. Let
\[
 F(z)=z+Az^2+Bz^3+Cz^4+\cdots\in\Sclass,
\]
and put $Q=A^2-B$ and $J=-A^3+2AB-C$.
Here $A,Q,J$ are initially complex. The odd exterior schlicht function
\[
 h(\zeta)=\frac1{\sqrt{F(1/\zeta^2)}}
 =\zeta+\frac{c_1}{\zeta}+\frac{c_3}{\zeta^3}
                 +\frac{c_5}{\zeta^5}+\cdots
\]
has coefficients
\[
 c_1=-\frac A2,\qquad c_3=\frac Q2-\frac{A^2}{8},\qquad
 c_5=\frac J2+\frac{AQ}{4}-\frac{A^3}{16}.
\]
Using the convention
\[
 \log\frac{h(\zeta)-h(\eta)}{\zeta-\eta}
       =-\sum_{m,n \ge 1}b_{mn}\zeta^{-m}\eta^{-n},
\]
we obtain $b_{11}=c_1$, $b_{13}=c_3$, and
\[
 b_{33}=c_5+c_1c_3+\frac{c_1^3}{3}
       =\frac J2-\frac{A^3}{24}.
\]
The Grunsky inequalities imply that the matrix
\begin{equation}\label{eq:M}
 M(A,Q,J)=
 \begin{pmatrix}
 -A/2&\sqrt3(Q/2-A^2/8)\\
 \sqrt3(Q/2-A^2/8)&3J/2-A^3/8
 \end{pmatrix}
\end{equation}
has operator norm at most one. This follows either by keeping rows and columns
$1,3$ in the strong Grunsky inequality, or from the corresponding complex
symmetric bilinear form.

\begin{lemma}\label{lem:support}
If $A$ is real, then for every real $\lambda$,
\begin{equation}\label{eq:support}
 \Real(J+\lambda Q) \ge  P_\lambda(A) \coloneq \frac{A^3}{12}+\frac{\lambda A^2}{4}
       +\frac{\lambda^2 A}{4}-\frac{\lambda^2}{2}-\frac23.
\end{equation}
Also, for~\eqref{eq:f}, with the parameter $a$ possibly complex,
\begin{equation}\label{eq:abound}
 \abs{a} \le b(s):=\frac23-\frac{s^2}{2}.
\end{equation}
\end{lemma}
\begin{proof}
Apply $\Real(v^TMv) \ge -\norm{v}^2$ with
$v=(\sqrt3\lambda/2,1)^T$ and multiply by $2/3$.
For the second assertion apply \eqref{eq:M} to $f$ itself. Its matrix is
\[
 \begin{pmatrix}0&\sqrt3s/2\\\sqrt3s/2&-3a/2\end{pmatrix}.
\]
A unitary diagonal congruence makes the off-diagonal entry and the bottom-right
entry non-negative real numbers without changing the norm. For a matrix $
\left(\begin{smallmatrix}0&b\\b&d\end{smallmatrix}\right)$ with $b,d \ge 0$,
having norm at most one is equivalent to $d \le 1-b^2$. Substitute
$b=\sqrt3s/2$, $d=3\abs{a}/2$.
\end{proof}

The same matrix constrains each omitted value.
\begin{corollary}\label{cor:omitted}
For a function of the form \eqref{eq:f}, every omitted value $w$ belongs to
\[
 K_{s,a}=\{w\ne0:\norm{M_{s,a}(w)} \le 1\},
\]
where
\begin{equation}\label{eq:omitted-matrix}
 M_{s,a}(w)=\begin{pmatrix}
 -1/(2w)&\sqrt3\bigl(s/2-1/(8w^2)\bigr)\\
 \sqrt3\bigl(s/2-1/(8w^2)\bigr)&-3a/2-1/(8w^3)
 \end{pmatrix}.
\end{equation}
\end{corollary}
\begin{proof}
Since $f(0)=0$, an omitted value is nonzero. The map
$F=f/(1-f/w)$ is normalized and univalent, and its coefficient functionals
are $A=1/w$, $Q=s$, and $J=-a$. Substitute in~\eqref{eq:M}.
\end{proof}

\section{The analytic bound}\label{sec:hand}

We will also use the classical Fekete--Szeg\H{o}
inequality~\cite{FeketeSzego}. We denote its parameter by $\mu$,
reserving $\lambda$ for the Grunsky supporting
inequality~\eqref{eq:support}.
In the notation $F(z)=z+Az^2+Bz^3+\cdots$ of
Section~\ref{sec:grunsky}, the Fekete--Szeg\H{o} inequality states
\[
 \abs{B-\mu A^2}
 \le 1+2\exp\left(-\frac{2\mu}{1-\mu}\right),
 \qquad\text{for $0 \le \mu < 1$}.
\]
For real $A$ and $Q=A^2-B$, taking real parts gives
\begin{equation}\label{eq:FS}
 \Real Q \ge (1-\mu)A^2-1
             -2\exp\left(-\frac{2\mu}{1-\mu}\right),
 \qquad\text{for $0 \le \mu < 1$}.
\end{equation}

\begin{proof}[Proof of Theorem~\ref{thm:hand}]
Set $R=592/1000$ and fix
\[
 t=997/1000, \qquad
 q= \left(\frac{t}{\artanh t}\right)^2, \qquad\text{and }
 A(\rho)=\frac t\rho+\frac{q\rho}{t}.
\]
Let $f\in\B$, write it as in \eqref{eq:f}, and let $w=\rho e^{i\theta}$ be any
omitted value with $\rho<t/\sqrt q$. The bounded-image transform gives a
function $F\in\Sclass$ for which
\begin{equation}\label{eq:boundedQJ}
 A=A(\rho),\qquad Q=q+t^2s e^{2i\theta},\qquad
 J=-t^3a e^{3i\theta}.
\end{equation}
For clarity, the construction is
\[
 f_\theta(z)=e^{-i\theta}f(e^{i\theta}z),\qquad
 g(z)=\frac{f_\theta(tz)}{t},\qquad
 F=\frac{g}{1-A(\rho)\,g+q\,g^2}.
\]
The Bloch inequality gives $\abs{f_\theta(z)} \le \artanh\abs{z}$, so
$\abs{g}<\artanh(t)/t=1/\sqrt q$ on $\D$. The denominator of~$F$ equals
$q(g-\rho/t)(g-t/(q\rho))$. Here $g$ omits $\rho/t$ because $f$ omits $w$, and
$t/(q\rho)>1/\sqrt q$ because $\rho<t/\sqrt q$. So~$F$ is holomorphic in $\D$,
with $F(0)=0$ and $F'(0)=1$. It is univalent, since $F(z_1)=F(z_2)$ gives
$(g(z_1)-g(z_2))(1-q\,g(z_1)g(z_2))=0$, while $\abs{q\,g(z_1)g(z_2)}<1$.
Expansion gives \eqref{eq:boundedQJ}.

\smallskip\noindent\emph{Small cubic: the Fekete--Szeg\H{o} inequality.}
If $s \le 196/1000$, then $D(0,R)\subset f(\D)$ already.
Indeed, the scalar comparisons in Appendix~\ref{app:scalar} give
\begin{equation}\label{eq:smallchecks}
 \tfrac{13}{4}e^{-5/8}<A(R),\qquad
 \tfrac92e^{-5/4}-1>q+t^2\frac{196}{1000}.
\end{equation}
If an omitted value had $\rho<R$, then
$A(\rho)>A(R)>\frac{13}{4}e^{-5/8}$. Now \eqref{eq:FS} with
$\mu=5/13$, for which $2\mu/(1-\mu)=5/4$, gives
\[
 \Real Q \ge \tfrac{8}{13}A(\rho)^2-1-2e^{-5/4}
        >\tfrac92e^{-5/4}-1>q+t^2s \ge \Real Q,
\]
a contradiction.

\smallskip\noindent
\emph{Remaining cubic ranges: Grunsky supporting forms.}
For larger $s$, we use the supporting inequality~\eqref{eq:support}, which holds
for every real $\lambda$. We choose $\lambda$ from the following table. Its last
two columns are rational lower bounds verified in Appendix~\ref{app:scalar}; the
quantities are defined immediately below.

\begin{center}
\begin{tabular}{ccccc}
\toprule
$l$&$u$&$\lambda$&$E_0(u)>$&$E_1(l)>$\\
\midrule
$196/1000$&$23/100$&$6$&$0.037$&$0.011$\\
$23/100$&$29/100$&$5$&$0.019$&$0.066$\\
$29/100$&$1/3$&$4$&$0.058$&$0.060$\\
\bottomrule
\end{tabular}
\end{center}
Here $s\in[l,u]$, $A_R=A(R)$, $b(s)=2/3-s^2/2$, and
\begin{align*}
 E_0(s)&=P_\lambda(A_R)-\lambda(q+t^2s),\\
 E_1(s)&=P_\lambda(A_R)-\lambda(q-t^2s)-4t^3b(s).
\end{align*}
The function $E_0$ decreases and $E_1$ increases with $s \ge 0$. Thus, the table
proves both $E_0$ and $E_1$ are positive throughout the corresponding range.

Write $a=u_4+iv_4$. The choice of center rests on the identity
\[
 \Real(ae^{3i\theta})-\bigl(u_4\cos\theta-3v_4\sin\theta\bigr)
 =-4\Real(ae^{i\theta})\sin^2\theta ,
\]
which follows from $\cos3\theta-\cos\theta=-4\cos\theta\sin^2\theta$ and
$\sin3\theta-3\sin\theta=-4\sin^3\theta$. A translation by $c$ contributes the
first harmonic $\Real(ce^{-i\theta})$. We choose $c$ so that, after the
tangent-line correction below, this harmonic is exactly
$t^3(u_4\cos\theta-3v_4\sin\theta)$; the quartic coefficient then enters only
through the factor $\sin^2\theta$. Define
\[
 H_\lambda=\frac{(A_R+\lambda)^2}{4}
                    \left(\frac t{R^2}-\frac qt\right),\qquad
 c=\frac{t^3}{H_\lambda}(u_4-3iv_4).
\]
The scalar checks give $H_\lambda>22$ for all three choices.
Consequently $\abs{c} \le 2t^3/H_\lambda<1/11$ and
$R+\abs{c} < 0.7 < t/\sqrt q$.

Suppose an omitted $w=\rho e^{i\theta}$ lies in $D(c,R)$.
The function $G_\lambda(\rho)=P_\lambda(A(\rho))$ is decreasing and
convex on $(0,t/\sqrt q)$, with $G_\lambda'(R)=-H_\lambda$.
This follows from $P_\lambda'(A)=(A+\lambda)^2/4>0$,
$P_\lambda''(A)=(A+\lambda)/2>0$, and $A''(\rho)>0$.
For $w$ in the open disk,
\[
 \rho=\Real(we^{-i\theta})
     =\Real\bigl((w-c)e^{-i\theta}\bigr)+\Real(ce^{-i\theta})
     <R+\Real(ce^{-i\theta}).
\]
The tangent inequality for the convex function $G_\lambda$ at $\rho=R$
and the identity
$H_\lambda\Real(ce^{-i\theta})=t^3(u_4\cos\theta-3v_4\sin\theta)$,
which is the definition of $c$, give
\[
 G_\lambda(\rho)>P_\lambda(A_R)
               -t^3(u_4\cos\theta-3v_4\sin\theta).
\]
Apply \eqref{eq:support} and \eqref{eq:boundedQJ}, and put
$x=\sin^2\theta$. A necessary condition is
\[
 E_0(s)+\bigl(2\lambda t^2s-4t^3\Real(ae^{i\theta})\bigr)x<0.
\]
By \eqref{eq:abound}, its left side is at least
\[
 (1-x)E_0(s)+xE_1(s)>0.
\]
This contradiction shows that in every case the open disk --- $D(0,R)$
in the small cubic range, $D(c,R)$ otherwise --- contains no omitted
value. So it lies in $f(\D)$, hence $R_f \ge  R$, and
Lemma~\ref{lem:normalization} completes the proof.
\end{proof}

Figure~\ref{fig:analytic} shows the proof for $\lambda=6$ in the polar
coordinates $w=\rho e^{i\theta}$. By \eqref{eq:support} and
\eqref{eq:boundedQJ}, an omitted value satisfies
$G_\lambda(\rho) \le \Real(J+\lambda Q)$, that is,
$\rho \ge \rho_\lambda(\theta)$ for an explicit curve $\rho_\lambda$. A disk
centered at the origin is a horizontal line in these coordinates, while the
boundary of $D(c,R)$ is, to first order, the curve
$R+\Real(ce^{-i\theta})$. The tangent inequality replaces $\rho_\lambda$ by
the lower bound $R+(G_\lambda(R)-\Real(J+\lambda Q))/H_\lambda$. Its
difference from $R+\Real(ce^{-i\theta})$ is the left side of the necessary
condition in the proof divided by $H_\lambda$, which is at least
$((1-x)E_0(s)+xE_1(s))/H_\lambda$.

\begin{figure}[ht]
\centering
\includegraphics[width=\textwidth]{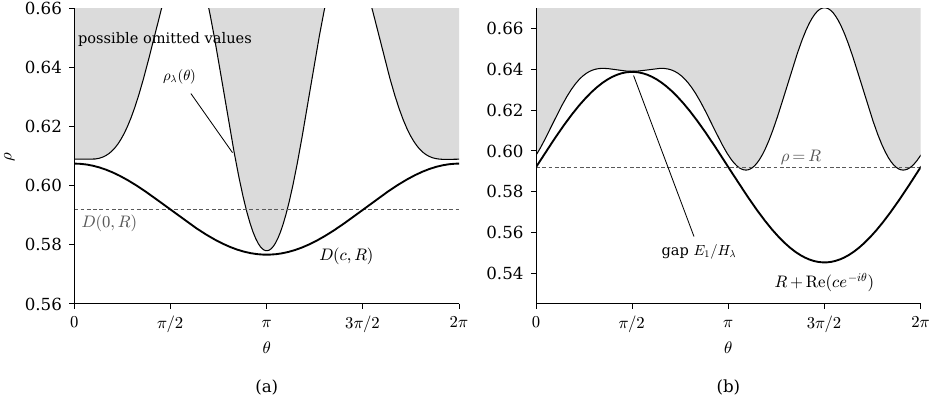}
\caption{The proof of Theorem~\ref{thm:hand} for $\lambda=6$, drawn in
the coordinates $(\theta,\rho)$ of $w=\rho e^{i\theta}$. Shaded: values of
$w$ not excluded by the supporting form. (a) $s=0.23$, $a=b(s)$, with the
exact constraint $\rho \ge \rho_\lambda(\theta)$. The disk $D(0,R)$ (dashed
line) meets the shaded region near $\theta=\pi$, where
$\rho_\lambda=0.578$; the translated disk $D(c,R)$, $c=0.0154$ (its
boundary is the solid curve), misses it. (b) $s=0.196$, $a=-ib(s)$, with
the tangent-line bound used in the proof and the curve
$R+\Real(ce^{-i\theta})$, $c=0.0467i$. The gap between them is smallest
at $\theta=\pi/2$, where it equals $E_1(s)/H_\lambda=2.8\cdot10^{-4}$, the
tightest entry of the table. All values listed here are rounded.}
\label{fig:analytic}
\end{figure}

\begin{remark}\label{rem:handcap} The constant $0.592$ is close to the limit of
this argument. The small-cubic step covers $s \le s_0(R,t)$, where $q+t^2s_0$ is
the largest value of the right-hand side of \eqref{eq:FS} at $A=A(R)$. It equals
$2e^{-2k}(2k+1)-1$, attained at $\mu=k/(1+k)$, where $2e^{-k}(1+k)=A(R)$.

The supporting forms must cover $[s_0,1/3]$. As $R$ grows, $s_0$ decreases while
the range of $s$ covered by each $\lambda$ shrinks. Numerical optimization
within this reduction, over $t\in[0.985,0.999]$ and real $\lambda$, suggests a
limiting radius near $0.5925$, with the best values for $t$ near $0.996$. This
motivates retaining more of the complex matrix inequality, which is the purpose
of Sections~\ref{sec:matrix} and~\ref{sec:certificate}.
\end{remark}

\section{An explicit disk from the complex matrix}\label{sec:matrix}

If we are willing to use some computer assistance, we can do better than $0.592$,
by keeping more information from the Grunsky matrix.
For $a=u+iv$ and $q=\abs{a}^2$, define a polynomial center $c(s,a)$ by
\begin{equation}\label{mat:eq:center}
  c(s,a) = s\bigl(A_0(s,q)u - iB_0(s,q)v\bigr),
\end{equation}
where $A_0(s,q) = \frac{2}{5}-\frac{q}{3}-\frac{s^2}{5}$, and
$B_0(s,q) = \frac{25}{24}-\frac{3q}{2}-\frac{s^2}{4}$.
Also, let the target radius be
\begin{equation}
  R =\frac{306121}{500000}=0.612242.\label{mat:eq:R}
\end{equation}
\begin{theorem}\label{mat:thm:disk}
For every $0 \le s \le 1/3$ and $\abs{a} \le b(s)$,
\[
 \overline D(c(s,a),R)\cap K_{s,a}=\varnothing.
\]

\end{theorem}

The translation is essential. In a floating-point scan of the parameter
set, the worst case for a disk centered at the origin is $s=1/3$,
$a=11/18=b(1/3)$. There, the largest disk centered at~$0$ that misses
$K_{s,a}$ has radius only about $0.568$, while the translated disk
$D(c,R)$ still misses $K_{s,a}$, see Figure~\ref{fig:kset}\,(a). For
$s=1/3$, $a=0$ the center is~$0$, and $R$ is limited by the obstruction
of Section~\ref{sec:obstruction}, see Figure~\ref{fig:kset}\,(b).

\begin{figure}[ht]
\centering
\includegraphics[width=\textwidth]{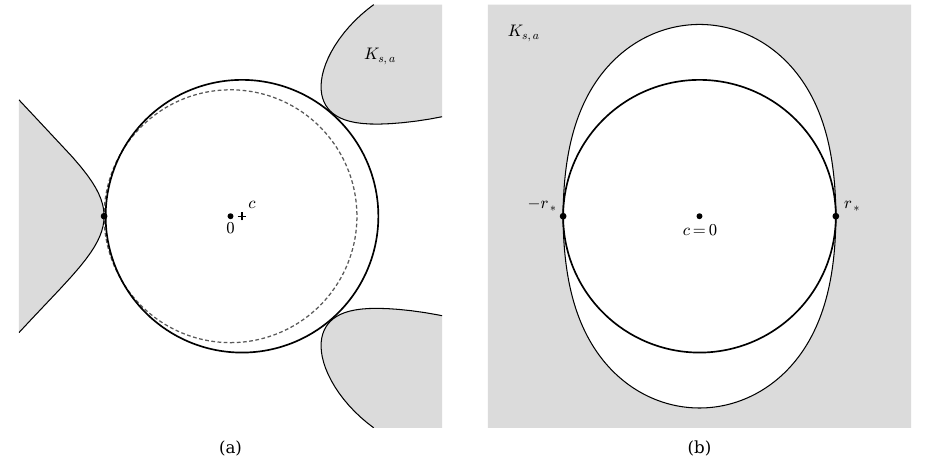}
\caption{The set $K_{s,a}$ (shaded) contains every omitted value
by Corollary~\ref{cor:omitted}. (a) $s=1/3$, $a=11/18$. The largest disk
centered at $0$ that misses $K_{s,a}$ has radius $0.5676$ (dashed). The
disk $D(c,R)$ of Theorem~\ref{mat:thm:disk}, with $c=c(s,a)=0.0516$ and
$R=0.612242$, also misses it (solid); in fact
$\dist(c,K_{s,a})=0.6161$. (b) $s=1/3$, $a=0$. Here $c=0$, and $D(0,R)$
nearly touches $K_{s,a}$ at $\pm r_*$, the obstruction of
Section~\ref{sec:obstruction}. The values are rounded.}
\label{fig:kset}
\end{figure}

\subsection{A polynomial radial bound}

Both weights in \eqref{mat:eq:center} are non-negative on the parameter set,
and are at most $25/24$. Hence
\[
 \abs{c} \le \frac{1}{3} \cdot \frac{25}{24} \cdot \frac{2}{3}=\frac{25}{108} \eqcolon L<R.
\]
For a direction $\theta$, set
\[
 h=\Real(ce^{-i\theta}),\qquad\text{and }
 t=(\Real c)\sin\theta-(\Imag c)\cos\theta.
\]
The positive radial endpoint of the disk is
$h+\sqrt{R^2-t^2}$. We use the following polynomial upper bound:
\begin{equation}\label{mat:eq:p}
 \rho \le p \coloneq R+h-\frac{t^2}{2R}-\frac{t^4}{8R^3}
 \quad\text{whenever}\quad
 \rho e^{i\theta}\in\overline D(c,R),\ \rho>0.
\end{equation}
Indeed, for $0 \le x \le 1$,
\[
 \sqrt{1-x} \le 1-\frac x2-\frac{x^2}{8},
\]
since the right side is positive and its square minus $1-x$ is
$x^3/8+x^4/64 \ge 0$. Apply this inequality with $x=t^2/R^2$.
Also
\[
 p \ge  R-L-\frac{L^2}{2R}-\frac{L^4}{8R^3}>0.
\]
The fourth-order term sharpens the bound $\rho \le R+h$ while keeping
the certificate rational in the circle coordinates.

\subsection{Complex quadratic forms and radial monotonicity}

For $w=\rho e^{i\theta}$, the unitary congruence with
$\operatorname{diag}(e^{i\theta/2},e^{3i\theta/2})$ changes $M_{s,a}(w)$ to
\begin{equation}\label{mat:eq:transformed}
 \widetilde M(\rho)=
 \begin{pmatrix}
 -1/(2\rho)&\sqrt3\bigl(se^{2i\theta}/2-1/(8\rho^2)\bigr)\\
 \sqrt3\bigl(se^{2i\theta}/2-1/(8\rho^2)\bigr)
 &-3ae^{3i\theta}/2-1/(8\rho^3)
 \end{pmatrix}.
\end{equation}
Take a fixed vector
\[
 z=(\alpha+i\beta,\sqrt3(\gamma+i\delta))^T
\]
and abbreviate
\begin{align*}
 A&=\alpha^2-\beta^2,& B&=\alpha\gamma-\beta\delta,
 &C&=\gamma^2-\delta^2,\\
 E&=\alpha\delta+\beta\gamma,& F&=2\gamma\delta,
 &N&=\alpha^2+\beta^2+3(\gamma^2+\delta^2).
\end{align*}
Let $d+ie=ae^{3i\theta}$, and put
\begin{align}
 V&=N+3s(B\cos2\theta-E\sin2\theta)
                  -\frac92(Cd-Fe),\label{mat:eq:V}\\
\shortintertext{and}
 T(\rho)&=\norm{z}^2+\Real\bigl(z^T\widetilde M(\rho)z\bigr)
 =V-\frac{A}{2\rho}-\frac{3B}{4\rho^2}-\frac{3C}{8\rho^3}.
 \label{mat:eq:T}
\end{align}
Note that the transpose in the formula for $T(\rho)$ is not a conjugate
transpose. Nevertheless, the contraction bound implies $\abs{z^T\widetilde Mz}
\le \norm{z}^2$, so $T(\rho) \ge 0$ is necessary for $\norm{M_{s,a}(w)} \le 1$.

\begin{lemma}[Monotone witness]\label{mat:lem:witness}
If $A,B,C \ge 0$ and
\begin{equation}\label{mat:eq:P}
 P:=8p^3V-4p^2A-6pB-3C<0,
\end{equation}
then every point $\rho e^{i\theta}$ with $0<\rho \le p$ is outside
$K_{s,a}$.
\end{lemma}
\begin{proof}
The derivative of \eqref{mat:eq:T} is
\[
 T'(\rho)=\frac{A}{2\rho^2}+\frac{3B}{2\rho^3}
                    +\frac{9C}{8\rho^4} \ge 0.
\]
As $p>0$, \eqref{mat:eq:P} says $T(p)<0$. Thus $T(\rho)<0$ for
$0<\rho \le p$, contradicting the necessary contraction condition.
\end{proof}

The complex quadratic form retains the imaginary part of the matrix:
$\beta$ and $\delta$ are allowed to be nonzero. The three elementary sign
conditions keep the argument monotone in the radius. No claim of radial
monotonicity for the full operator norm is needed.

\section{The interval certificate}\label{sec:certificate}

The center satisfies $c(s,\bar a)=\overline{c(s,a)}$ and $c(s,-a)=-c(s,a)$.
Correspondingly $K_{s,\bar a}=\overline{K_{s,a}}$, since
$M_{s,\bar a}(\bar w)=\overline{M_{s,a}(w)}$, and $K_{s,-a}=-K_{s,a}$, since,
with $D=\operatorname{diag}(1,-1)$,
\[
 M_{s,-a}(-w)=-D M_{s,a}(w)D.
\]
These symmetries reduce to $u,v \ge 0$, with all directions still included.
For the transformed matrix, they act by conjugation or the identity.
Replacing a witness by its conjugate preserves its dyadic coordinates
and the conditions $A,B,C \ge 0$, so the same reduction applies to the
finite certificate below. The verifier uses
\[
 s=x_0/3,\qquad a=b(s)x_1 e^{i\pi x_2/2},\qquad
 \theta=2\pi x_3,\qquad \text{where $x\in[0,1]^4$.}
\]

\begin{lemma}[Finite certificate]\label{lem:finite-certificate} For $R$ and
$c(s,a)$ in \eqref{mat:eq:center}--\eqref{mat:eq:R}, every $0 \le s \le 1/3$,
$\abs{a} \le b(s)$ and real $\theta$ admits numbers
$\alpha,\beta,\gamma,\delta\in 2^{-16}\mathbb Z$ such that $A,B,C \ge 0$ and
$P<0$, with the notation of \eqref{mat:eq:V}--\eqref{mat:eq:P}.
\end{lemma}

The proof is an exhaustive verification by bisection of $[0,1]^4$. On each
box, a dyadic witness is proposed in floating point, from a smallest-eigenvalue
vector of the real form of~\eqref{mat:eq:transformed} at the midpoint, and is
then held fixed on the whole box. The verifier checks $A,B,C\ge0$ exactly and
encloses $P$ over the box by interval evaluation and a mean-value form, with
derivatives from interval automatic differentiation. A box is accepted only if
the computed upper bound for $P$ is strictly negative; otherwise it is bisected.
No eigenvalue, error estimate or sign returned by the floating-point
computation is used to accept a box, and no consistency between the witnesses
on different boxes is required. All interval computations use Arb ball
arithmetic~\cite{Arb} at $80$-bit precision.

The exhaustive run finished successfully after visiting $187{,}263$ boxes, of
which $93{,}632$ are terminal, with maximal depth $49$, and took about $42$
seconds on the author's laptop. As a check against programming errors, an
independently written verifier, using exact integer-dyadic interval arithmetic
and a rational parametrization of the angles, reproduces the result. The source
code of both verifiers, run records with hashes of all terminal boxes and their
witnesses, and further self-checks are in the supplementary
archive~\cite{WikArchive}.

These computations prove Lemma~\ref{lem:finite-certificate}.
Lemma~\ref{mat:lem:witness} and~\eqref{mat:eq:p} then exclude the entire closed
disk, proving Theorem~\ref{mat:thm:disk}. By Corollary~\ref{cor:omitted}, the
open disk $D(c,R)$ contains no omitted value, so it lies in $f(\D)$; the
rotation reduction of Section~\ref{sec:normalization} gives
Theorem~\ref{thm:computer}(a).

\section{Beyond the \texorpdfstring{$2\times2$}{2x2} matrix}
\label{sec:hybrid}

This section proves part~(b) of Theorem~\ref{thm:computer}. We first show that
the radius of part~(a) is essentially the limit of the $2\times2$ matrix with
the center~\eqref{mat:eq:center}, and that this limit is a local phenomenon.
We then add two ingredients, a third Grunsky row and pointwise bounds from the
Bloch inequality, and combine them in one type of certificate.

\subsection{An obstruction for the chosen center}\label{sec:obstruction}

There is a concrete obstruction for the explicit center
formula~\eqref{mat:eq:center}.
Take $s=1/3$, $a=0$, so that $c=0$, and take $w=r>0$ real. Then
\[
 192r^4\det(I+M_{1/3,0}(r))
   =176r^4-96r^3+24r^2-24r+3 \eqcolon F(r).
\]
Let $r_*$ be the unique root of $F$ in $(0.6,0.7)$. Exact rational
substitution brackets it by
\[
 0.612242772171<r_*<0.612242772172.
\]
For uniqueness, $F'(0.6)>0$ and
$F''(r)=2112r^2-576r+48>0$ on the real line.
At this root $-1$ is an eigenvalue of the real symmetric matrix
$M_{1/3,0}(r_*)$. The other eigenvalue is
\[
 1-\frac1{2r_*}-\frac1{8r_*^3}\in(-0.37,-0.35),
\]
so $\norm{M_{1/3,0}(r_*)}=1$. In particular $r_*\in K_{1/3,0}$.

Thus no uniform radius larger than $r_*$ can work with \eqref{mat:eq:center};
the certificate of Section~\ref{sec:certificate} is within $7.72172\cdot10^{-7}$ of that ceiling. This is
a ceiling for this \emph{center rule}, not an upper bound for $\U$ or a proof of
optimality among all centers. The pair $s=1/3$, $a=0$ is attained, for instance
by $\arctan z\in\B$, whose image is the strip $\abs{\Real w}<\pi/4$. So the
obstruction comes from the relaxation $K_{s,a}$ and not from an actual omitted
value. It is also local: with a center chosen separately on each parameter box,
the $2\times2$ matrix alone excludes a disk of radius $619/1000$ outside a set
of $6.5\%$ of the volume of the parameter cube of Section~\ref{hyb:sec:covering}
(type (i) there).

To go beyond $r_*$, we sharpen the relaxation near $s=1/3$, $a=0$. The
next Grunsky row involves the coefficients $a_5$ and $a_6$. For $\arctan z$,
where $a_5=1/5$ and $a_6=0$, the $3\times3$ compression below has norm about
$1.38$ at $w=r_*$, so it excludes the point $r_*$. For other admissible values
of $a_5$ and $a_6$ the relaxation still reaches close to the origin, so we let
the center depend on these coefficients through a fixed partition of their
range. We also add pointwise consequences of the Bloch inequality, which rule
out parts of the relaxation near the edge of the parameter set.

\subsection{\texorpdfstring{The $3\times3$ compression}{The 3x3 compression}}
For $f$ as in \eqref{eq:f}, with $f(z)=z-sz^3+az^4+a_5z^5+a_6z^6+\cdots$, put
\[
 \varphi(\zeta)=\frac1{f(1/\zeta)}
 =\zeta+\frac s\zeta-\frac a{\zeta^2}+\frac{\beta_3}{\zeta^3}
   +\frac{\beta_4}{\zeta^4}+\cdots,
 \qquad
 \beta_3=s^2-a_5,\quad \beta_4=-a_6-2sa.
\]
For $u\in\C$, the coefficients $b_{mn}$ of the odd function
$h_u(\zeta)=\sqrt{\varphi(\zeta^2)-u}$, in the convention of
Section~\ref{sec:grunsky}, satisfy
\begin{align*}
 b_{11}&=-\tfrac u2, & b_{13}&=\tfrac s2-\tfrac{u^2}8,
 & b_{15}&=-\tfrac a2+\tfrac{us}4-\tfrac{u^3}{16},\\
 b_{33}&=-\tfrac a2-\tfrac{u^3}{24}, &
 b_{35}&=\tfrac{\beta_3}2+\tfrac{s^2}8+\tfrac{u^2s}{16}-\tfrac{3u^4}{128},
 & b_{55}&=\tfrac{\beta_4}2-\tfrac{sa}2-\tfrac{us^2}8+\tfrac{u^3s}{16}
   -\tfrac{9u^5}{640},
\end{align*}
and $b_{22}=s/2$, $b_{24}=-a/2$, $b_{44}=\beta_3/2+s^2/4$. Put
\[
 G_u=\bigl(\sqrt{mn}\,b_{mn}\bigr)_{m,n\in\{1,3,5\}},\qquad
 E=\bigl(\sqrt{mn}\,b_{mn}\bigr)_{m,n\in\{2,4\}}
  =\begin{pmatrix}s&-\sqrt2\,a\\-\sqrt2\,a&s^2+2\beta_3\end{pmatrix}.
\]
The upper left $2\times2$ block of $G_{1/w}$ is the matrix
$M_{s,a}(w)$ of~\eqref{eq:omitted-matrix}.

\begin{lemma}\label{hyb:lem:grunsky}
Let $f\in\Sclass$ be as in \eqref{eq:f}. Then $\norm{G_0}\le1$ and
$\norm{E}\le1$, and $\norm{G_{1/w}}\le1$ for every omitted value $w$.
Moreover
\[
 \abs{\beta_3}\le\frac2{\sqrt{15}}+\frac1{36}<\frac9{16},\qquad
 \abs{\beta_4}\le\frac{28}{45}<\frac58.
\]
\end{lemma}
\begin{proof}
For $u=0$ put $F=f$, and for $u=1/w$ with $w$ omitted let $F=f/(1-f/w)$
be the map in the proof of Corollary~\ref{cor:omitted}. In both cases
$F\in\Sclass$ and $1/F(1/\zeta)=\varphi(\zeta)-u$. Thus $h_u$ is the odd square-root transform
of a function in $\Sclass$, and the strong Grunsky inequality applies to its
coefficients. Since $h_u$ is odd, $b_{mn}=0$ when $m+n$ is odd, so $G_u$ and
$E$ are compressions of a contraction. The even entries do not depend on~$u$.
The expansions are routine; the archive~\cite{WikArchive} contains a symbolic
derivation. For the last claim use $\sqrt{15}\abs{b_{35}}\le1$ and
$5\abs{b_{55}}\le1$ at $u=0$, with $s\le1/3$ and $\abs a\le2/3$.
\end{proof}

Consequently every omitted value lies in
\[
 K_{s,a,\beta}=\{w\ne0:\norm{G_{1/w}}\le1\}\subset K_{s,a}.
\]
Every point of $K_{s,a}$, and hence of $K_{s,a,\beta}$, satisfies
$\abs w>1/2$. Indeed, the first row of $M_{s,a}(w)$ is
$\bigl(-1/(2w),\sqrt3(s/2-1/(8w^2))\bigr)$. For $\abs w<1/2$ its first entry
has modulus greater than one, and for $\abs w=1/2$ the row has norm at least
$\bigl(1+3(1/2-s/2)^2\bigr)^{1/2}>1$, since $s\le1/3$.

\subsection{Point bounds from the Bloch inequality}
The second ingredient is elementary.

\begin{lemma}\label{hyb:lem:bloch}
Let $f\in\B$ be as in \eqref{eq:f}, and let
\[
 P(\zeta)=1-3s\zeta^2+4a\zeta^3+5(s^2-\beta_3)\zeta^4
   -6(\beta_4+2sa)\zeta^5
\]
be the Taylor polynomial of $f'$ of degree $5$. For $\zeta\in\D$,
\[
 \abs{P(\zeta)}\le C(\abs\zeta)\coloneq\frac1{1-\abs\zeta^2}
   +\sum_{k\ge6}\gamma_k\abs\zeta^k,\qquad
 \gamma_k=\frac{k+2}{2}\Bigl(1+\frac2k\Bigr)^{k/2}.
\]
\end{lemma}
\begin{proof}
Write $f'(z)=\sum_k c_kz^k$. Since $\abs{f'(z)}\le1/(1-r^2)$ on
$\abs z=r$, Cauchy's estimate gives $\abs{c_k}\le r^{-k}/(1-r^2)$. The
choice $r^2=k/(k+2)$ gives $\abs{c_k}\le\gamma_k$. Now apply the triangle
inequality to $f'(\zeta)-\sum_{k\ge6}c_k\zeta^k$.
\end{proof}

Since $(1+2/k)^{k/2}<e$, the tail $\sum_{k\ge K}\gamma_kt^k$ is bounded by
$\frac e2\sum_{k\ge K}(k+2)t^k$, which has a closed form; so $C(t)$ can be
enclosed rigorously. The bound is useful only for small $\abs\zeta$. We use
it at $70$ fixed points with $0.15\le\abs\zeta\le0.35$, near the imaginary
axis. Lemma~\ref{hyb:lem:bloch} is the only place where the proof of
Theorem~\ref{thm:computer}(b) uses more than $f\in\Sclass$, $a_2=0$ and
$\abs{a_3}\le1/3$.

\subsection{Certificates}
For a complex $n\times n$ matrix $M$ let $\widehat M$ be its real
$2n\times2n$ form, with blocks $\Real M$, $-\Imag M$ in the first row and
$\Imag M$, $\Real M$ in the second. Then $\norm M\le1$ if and only if
\[
 \Phi(M)=\begin{pmatrix}I&\widehat M\\\widehat M^T&I\end{pmatrix}
\]
is positive semidefinite. Thus $\norm M\le1$ implies $y^T\Phi(M)y\ge0$ for
every real vector $y$. We use three kinds of necessary conditions: (K)
constrains a putative omitted value $w$, (G) constrains the coefficients
independently of $w$, and (P) adds the information from Landau's
normalization. For a function $f\in\B$ as in \eqref{eq:f} and an
omitted value $w$, the following inequalities hold:
\begin{enumerate}
\item[(K)] $y^T\Phi(G_{1/w})y\ge0$, for real $y\in\mathbb R^{12}$;
\item[(G)] $y^T\Phi(G_0)y\ge0$ and $y'^T\Phi(E)y'\ge0$, for real
 $y\in\mathbb R^{12}$, $y'\in\mathbb R^8$;
\item[(P)] $\mu\,C(\abs\zeta)-\Real\bigl(\bar\nu P(\zeta)\bigr)\ge0$,
 for $\zeta\in\D$ and $\nu\in\C$ with $\abs\nu\le\mu$.
\end{enumerate}
Each left-hand side is a real affine function of
$(\Real\beta_3,\Imag\beta_3,\Real\beta_4,\Imag\beta_4)$, whose
coefficients depend polynomially on $s$, $a$, $\bar a$ and, for (K), on $1/w$
and $1/\bar w$.

The certificate of Section~\ref{sec:certificate} is of this form. For a
complex $n\times n$ matrix $M$ and $z=\xi+i\eta$ with $\xi,\eta\in\mathbb R^n$,
the vector $y=(\xi,-\eta,\xi,\eta)$ satisfies
\[
 y^T\Phi(M)y=2\bigl(\norm z^2+\Real(z^TMz)\bigr).
\]
Since $z^T\widetilde M(\rho)z=(Dz)^TM_{s,a}(w)(Dz)$ with
$D=\operatorname{diag}(e^{i\theta/2},e^{3i\theta/2})$, the inequality
$T(\rho)\ge0$ of~\eqref{mat:eq:T} is one half of an inequality of type (K),
whose vector vanishes in the rows of index~$5$. Lemma~\ref{mat:lem:witness}
adds radial monotonicity, so that one such inequality excludes a whole radial
segment. Below, the polar cells play that role, and the certificate may
combine several inequalities of all three types.

\begin{definition}\label{hyb:def:cert}
Let $Q$ be a compact set of pairs $(s,a)$, let $\Pi\subset\C^2$ be a compact
box of pairs $(\beta_3,\beta_4)$, and let $W\subset\C$ be compact. A
\emph{certificate} for $(Q,\Pi,W)$ is a finite family of inequalities of
types (K), (G) and (P) whose sum $L(s,a,\beta,w)$ satisfies $L<0$ for all
$(s,a)\in Q$, $\beta\in\Pi$ and $w\in W$. An \emph{exclusion certificate} for
$(Q,\Pi)$ is a finite family of inequalities of types (G) and (P) whose sum
is negative for all $(s,a)\in Q$ and $\beta\in\Pi$.
\end{definition}

\begin{lemma}\label{hyb:lem:cert}
Let $f\in\B$ be as in \eqref{eq:f}, with $(s,a)\in Q$ and
$(\beta_3,\beta_4)\in\Pi$. If $(Q,\Pi)$ has an exclusion certificate, then
no such $f$ exists. If $(Q,\Pi,W)$ has a certificate, then $W$ contains no
omitted value of~$f$.
\end{lemma}
\begin{proof}
By Lemmas~\ref{hyb:lem:grunsky} and~\ref{hyb:lem:bloch}, at the parameters of
$f$ every inequality in the family holds, and for (K) this uses that $w$ is
omitted. Hence their sum is non-negative, a contradiction.
\end{proof}

Since $L$ is affine in $\beta$, its maximum over a box is explicit: if
$L=\ell_0+\sum_j\ell_jb_j$ in the four real coordinates $b_j$ of $\beta$ and
$\Pi=\prod_j[m_j-h_j,m_j+h_j]$, then
\[
 \max_{\Pi}L=\ell_0+\sum_j\bigl(\ell_jm_j+h_j\abs{\ell_j}\bigr).
\]
The coefficients $\ell_j$ depend on $(s,a)$ and $w$, and the verifier
encloses this expression in interval arithmetic over $Q$ and $W$. All data in the certificates, namely the
vectors $y$, the points $\zeta$ and the multipliers $\mu$ and $\nu$, are
dyadic rationals and are held fixed on the whole set $Q\times\Pi\times W$.
The exception is the form of Section~\ref{sec:matrix}, used for types (i)
and (ii) of Section~\ref{hyb:sec:covering}: there the fixed dyadic datum is
the vector $z$, and the vector of the corresponding inequality of type (K) is
the real form of $Dz$, which depends on~$\theta$.

\subsection{The covering}\label{hyb:sec:covering}
We use the coordinates $x\in[0,1]^3$ of Section~\ref{sec:certificate},
\[
 s=x_0/3,\qquad a=b(s)x_1e^{i\pi x_2/2}.
\]
Replacing $f(z)$ by $\overline{f(\bar z)}$ or by $-f(-z)$ preserves $\B$ and
$R_f$, and changes $a$ to $\bar a$ or to $-a$. So we may assume that $a$ lies
in the closed first quadrant. Then $\abs a\le b(s)$ by~\eqref{eq:abound}
gives $x\in[0,1]^3$. Let
$\Pi_0=\{\abs{\Real\beta_j},\abs{\Imag\beta_j}\le h_j\}$ with $h_3=9/16$ and
$h_4=5/8$. By Lemma~\ref{hyb:lem:grunsky}, $(\beta_3,\beta_4)\in\Pi_0$.
Split $\Pi_0$ into twelve boxes $\Pi_1,\dots,\Pi_{12}$ by cutting $\Real\beta_3$
at $0$ and $1/5$ and $\Imag\beta_3$ at $-1/5$, $0$ and~$1/5$. (In the
computation the cuts at $\pm1/5$ are at the nearest double-precision numbers;
any cut points give a partition.)

For a center $c$ and $R=619/1000$, the set $\overline D(c,R)\cap\{\abs w>1/2\}$
is covered by finitely many \emph{polar cells}, i.e.\ cells of the form
$\{\rho e^{i\theta}:\theta\in[\theta_0,\theta_1],\,\rho\in[\rho_0,\rho_1]\}$;
the remaining points of the disk are not in $K_{s,a,\beta}$.
Call a closed box $Q$ in the $x$-coordinates \emph{good} if one of the
following holds.
\begin{enumerate}
\item[(i)] There is a center $c$ such that every polar cell $W$ of such a
 covering has a certificate for $(Q,\Pi_0,W)$ that uses only rows $1$ and $3$
 in type~(K), and no inequalities of types (G) and (P).
\item[(ii)] The box $Q\times[0,1]$, in the four coordinates of
 Section~\ref{sec:certificate}, is a finite union of boxes, each with a fixed
 witness as in Lemma~\ref{lem:finite-certificate}, for the center $c(s,a)$
 of~\eqref{mat:eq:center} and with $R=619/1000$ in place of $0.612242$.
\item[(iii)] $(Q,\Pi_0)$ has an exclusion certificate.
\item[(iv)] For each $k=1,\dots,12$, the box $Q$ is a union of closed boxes
 $Q'$, each of which either has an exclusion certificate for $(Q',\Pi_k)$, or
 has a center $c$ such that every polar cell $W$ of the covering of
 $\overline D(c,R)\cap\{\abs w>1/2\}$ has a certificate for
 $(Q',\Pi_k,W)$.
\end{enumerate}

\begin{lemma}\label{hyb:lem:finite}
The cube $[0,1]^3$ is a finite union of good boxes.
\end{lemma}

\begin{proof}[Proof of Theorem~\ref{thm:computer}(b)]
Let $f\in\B$, normalized as above, and let $x\in[0,1]^3$ be its
coordinates. By Lemma~\ref{hyb:lem:finite}, $x$ lies in a good box $Q$.
In case (i), no polar cell contains an omitted value, by
Lemma~\ref{hyb:lem:cert}, and the rest of the disk lies in $\abs w\le1/2$. In
case (ii), the argument of Section~\ref{sec:certificate} shows that
$\overline D(c(s,a),R)$ misses $K_{s,a}$, which contains every omitted value.
Case (iii) cannot occur, by Lemma~\ref{hyb:lem:cert}. In case
(iv), $(\beta_3,\beta_4)$ lies in some $\Pi_k$, and $x$ lies in some box
$Q'$ of the corresponding subdivision. The box $Q'$ cannot carry an exclusion
certificate, so the disk $\overline D(c,R)$ contains no omitted value. In
every case the open disk $D(c,R)$ contains no omitted value, so it lies in
$f(\D)$, as noted in Section~\ref{sec:normalization}, and $R_f\ge R$.
\end{proof}

Lemma~\ref{hyb:lem:finite} is verified by computer, and
Table~\ref{hyb:tab:stats} summarizes the verification. Type (ii) is the verifier of Section~\ref{sec:certificate} at the larger
radius. Boxes of type (iii) lie in the region $x_0,x_1\ge5/6$, $x_2\le1/2$.
There $s$ and $\abs a$ are both close to their maximal values, and no
$f\in\B$ has such coefficients. Type (iv) is needed only for $x_0\ge43/48$ and
$x_1\le7/16$, that is, for $s>0.298$ and $\abs a<0.44\,b(s)$, around the
obstruction of Section~\ref{sec:obstruction}.

\begin{table}[ht]
\centering
\small
\begin{tabular}{llrr}
\toprule
Type & Inequalities & Boxes & Check\\
\midrule
(i) & (K), rows $1,3$ & $69{,}087$ & $32$ s\\
(ii) & Lemma~\ref{mat:lem:witness} & $111{,}865$ & $5$ s\\
(iii) & (G), (P) & $91$ & $<1$ s\\
(iv) & (K), (G), (P) & $532\times12$ & $7$ min\\
\bottomrule
\end{tabular}
\vspace*{8pt}

\caption{The certificate for Lemma~\ref{hyb:lem:finite}. The boxes of type
(iv) carry $15{,}771$ disks and $875$ exclusion certificates, with about
$10^7$ polar cells and $9.7\cdot10^5$ distinct certificates. Times are wall
clock for checking the stored certificates on four cores of an Apple M4~Pro.}
\label{hyb:tab:stats}
\end{table}

The centers, the partition of $\Pi_0$ and the certificates were found by
floating-point computations, the certificates from dual solutions of
semidefinite programs. Finding them took several hundred core-hours.
None of this is part of the proof. As in Section~\ref{sec:certificate}, every inequality used to accept a box is
checked in Arb ball arithmetic, with all data held fixed as dyadic
rationals. The certificates of all four types are stored, and a separate checker
replays them without any floating-point search. An exact check on the
stored box endpoints confirms that the good boxes cover $[0,1]^3$. The
verifiers, the certificate data and the coverage check are in the
archive~\cite{WikArchive}.

\section{An improved upper bound}\label{sec:upper}

For a domain $\Omega\subset\C$ and $w\in\Omega$ let $r(\Omega,w)$ be the
conformal radius, so that $r(f(\D), f(0)) = \abs{f'(0)}$ for univalent $f$.
Let $\rho(\Omega)$ be the inradius, and let $\capa F$ be the logarithmic
capacity of a compact set $F$. Since $R_f=\rho(f(\D))$ and both quantities
scale linearly,
\begin{equation}\label{up:eq:reform}
 \U=\inf\Bigl\{\frac{\rho(\Omega)}{r(\Omega,w)}:\ \Omega\subsetneq\C
 \text{ simply connected},\ w\in\Omega\Bigr\}.
\end{equation}
In particular, every simply connected domain $D \ni 0$ gives $\U \le
\rho(D)/r(D,0)$. All known upper bounds come from domains of Goodman
type~\cite{Robinson,Goodman,BellerHummel,CarrollOrtegaCerda,
OrtegaCerdaPridhnani}: radial slits from the cube roots of unity and from twice
the cube roots of $-1$, completed by further boundary arcs that keep the
inradius at~$1$. Carroll~\cite{Carroll08} showed that a boundary arc of an
extremal domain that touches no extremal disk is harmonically symmetric with
respect to the origin, and Carroll and Ortega-Cerd\`a~\cite{CarrollOrtegaCerda}
used this condition, together with Fedorov's solution of a four-point
P\'olya--Chebotarev problem, to construct the best rigorous bound so far. We
keep the first two Goodman generations, insert five free arcs in each sector of
opening $\pi/3$, and truncate by a circle. The free arcs are designed as
critical vertical trajectories of a quadratic differential, but the proof does not use
this: the arcs in the theorem are explicit polynomial curves, and every estimate
is checked in interval arithmetic.

\subsection{The quotient map}
Fix $R>2$, put $\omega=e^{2\pi i/3}$, and let
\[
 \Gamma_R=\bigcup_{k=0}^2\omega^k\Bigl([1,R]\cup e^{i\pi/3}[2,R]\Bigr)
\]
be the first two Goodman generations inside $\abs w<R$. Define
\[
 \Phi_R(t)=\frac{R^6+1-t-R^6/t}{\Delta_R},\qquad
 \Delta_R=\frac98R^6+9,\qquad
 m_\pm=\frac{(R^3\pm1)^2}{\Delta_R}.
\]

\begin{lemma}\label{up:lem:quotient}
\begin{enumerate}
\item $\Phi_R$ maps $A_R=\{0<\abs t<R^3\}\setminus([1,R^3]\cup[-R^3,-8])$
conformally onto $\C\setminus[0,1]$. It maps $[1,R^3]$, the circle
$\abs t=R^3$ and $[-R^3,-8]$ onto $[0,m_-]$, $[m_-,m_+]$ and $[m_+,1]$,
respectively. Moreover $\Phi_R(\bar t)=\overline{\Phi_R(t)}$, $\Imag\Phi_R(t)$
has the sign of $\Imag t$, and $\Phi_R(t)=-R^6/(\Delta_Rt)+O(1)$ as $t\to0$.
\item Let $\gamma_1,\dots,\gamma_n$ be pairwise disjoint arcs in the closed
upper half-plane such that $\gamma_j\cap\mathbb R=\{b_j\}$ is one endpoint,
with $0<b_j<1$. Put $T=[0,1]\cup\bigcup_j(\gamma_j\cup\bar\gamma_j)$ and
\[
 D_T=\{0\}\cup\bigl\{w : w^3\in A_R, \Phi_R(w^3)\notin T\bigr\}.
\]
Then $D_T$ is a simply connected domain containing $0$, invariant under
$w\mapsto\omega w$ and $w\mapsto\bar w$, and
\begin{equation}\label{up:eq:radius}
 r(D_T,0)^3=\frac{R^6}{\Delta_R\capa T}.
\end{equation}
Explicitly, $D_T$ is the disk $\abs{w} < R$ minus $\Gamma_R$ and minus the
six arcs $\omega^k\lambda_j$, $\omega^k\bar\lambda_j$ $(k=0,1,2)$ for
each $j$, where $\lambda_j$ is the lift of $\Phi_R^{-1}(\gamma_j)$ to the
sector $0 \le \arg w \le \pi/3$.
\end{enumerate}
\end{lemma}

\begin{proof}
(1) Since $\Phi_R(t)-\Phi_R(t')=(t'-t)(1-R^6/(tt'))/\Delta_R$, the rational
function $\Phi_R$ of degree two identifies exactly the pairs $t$ and $R^6/t$.
The punctured disk $0<\abs t<R^3$ is a fundamental domain for this involution
apart from its boundary circle, so $\Phi_R$ maps it bijectively onto
$\C\setminus\Phi_R(\{\abs t=R^3\})$. On the circle,
$\Phi_R(R^3e^{i\phi})=(R^6+1-2R^3\cos\phi)/\Delta_R$ fills $[m_-,m_+]$. On the
real axis $\Phi_R'(t)=(R^6/t^2-1)/\Delta_R>0$ for $\abs t<R^3$, and
$\Phi_R(1)=0$, $\Phi_R(-8)=1$, so the two slits map onto $[0,m_-]$ and
$[m_+,1]$. The identity $\Imag\Phi_R(t)=\Imag t\,(R^6/\abs{t}^2-1)/\Delta_R$ and
the expansion at $0$ are immediate.

(2) Let $\Omega_T=\{0\}\cup(A_R\cap\Phi_R^{-1}(\C\setminus T))$. By (1),
$\Phi_R$ maps $\Omega_T\setminus\{0\}$ conformally onto $\C\setminus T$ and
sends $0$ to $\infty$. Since $T$ is a finite tree, $\Chat\setminus T$ is a
simply connected domain, and hence so is $\Omega_T$.
Let $g$ map $\Chat\setminus T$ onto $\Chat\setminus\overline\D$ with
$g(z)=z/\capa T+O(1)$ at $\infty$. Then $1/g\circ\Phi_R$ maps $\Omega_T$ onto
$\D$, fixes $0$, and has derivative of modulus $\Delta_R\capa T/R^6$ there;
hence $r(\Omega_T,0)=R^6/(\Delta_R\capa T)$. If $F$ maps $\D$ onto $\Omega_T$
with $F(0)=0$ and $F'(0)>0$, then $F(z)/z$ has no zeros, and
$f(z)=z\bigl(F(z^3)/z^3\bigr)^{1/3}$ maps $\D$ conformally onto
$D_T=\{w:w^3\in\Omega_T\}$ with $f'(0)=F'(0)^{1/3}$. This proves
\eqref{up:eq:radius}. The description of $D_T$ follows from (1), since
$w^3\in[1,R^3]$ or $w^3\in[-R^3,-8]$ exactly when $w\in\Gamma_R$.
\end{proof}

For $T=[0,1]$ this gives $r(D_{[0,1]},0)^3=4R^6/\Delta_R$, which increases to
$32/9$ as $R\to\infty$, the value for the plane minus the six Goodman rays. The
circle $\abs w=R$ costs very little, and the remaining task is to choose arcs
$\gamma_j$ that bring the inradius down to $1$ while increasing $\capa T$ as
little as possible.

\subsection{A capacity bound for almost critical arcs}

For tips $a_j=x_j+iy_j$ $(y_j>0)$ and bases $b_j\in(0,1)$ put
\[
 P(z)=\prod_{j=1}^n(z-b_j),\qquad
 A(z)=z(z-1)\prod_{j=1}^n(z-a_j)(z-\bar a_j),\qquad
 q=\frac P{\sqrt A},
\]
and
\[
 p(s)=\prod_j(1-b_js),\qquad a(s)=\prod_j\bigl((1-x_js)^2+y_j^2s^2\bigr),
 \qquad N(s)=\frac{p(s)^2-a(s)}s .
\]
Since $p(0)=a(0)=1$, $N$ is a polynomial. If the arcs $\gamma_j$ from $b_j$ to
$a_j$ are exactly critical vertical trajectories of $q(z)^2dz^2$, so that
$\Real\int q\,dz$ is constant along them, then the Green function
of $\Chat\setminus T$ with pole at infinity is $\Real\int_1^zq$. The bases $b_j$
are then determined by $n$ conditions that are linear in the coefficients of
$P$, which is how the arcs were found, see Section~\ref{up:sec:search}. The
following lemma requires no such exactness: the failure of the arcs to be
critical costs at most $2\sum_j\varepsilon_j$ in $\log\capa T$.

\begin{lemma}\label{up:lem:cap}
Let $T$ be as in Lemma~\ref{up:lem:quotient}(2), where $\gamma_j$ is a
rectifiable arc from $b_j$ to $a_j$, and the points $0,1,b_j,a_j$ are
distinct. Choose the branch of $q$ on $\C\setminus T$ that is positive on
$(1,\infty)$, and let $\varepsilon_j \ge \int_{\gamma_j}\abs{\Real(q\,dz)}$.
Then
\[
 \log\capa T \le -\log4-J+2\sum_{j=1}^n\varepsilon_j,\qquad
 J=2\int_0^1\frac{N(s)}{\sqrt{a(s)}\,\bigl(p(s)+\sqrt{a(s)}\bigr)}\,du,
 \quad s=1-u^2 .
\]
\end{lemma}

\begin{proof}
All zeros of $A$ lie on the continuum $T$. A closed curve in $\C\setminus T$
therefore encloses either none of them or all $2n+2$ of them, so $q$ has a
single-valued branch on $\C\setminus T$. Since $q(z)=1/z+O(z^{-2})$ at
infinity, every period of $q\,dz$ in $\C\setminus T$ is an integer multiple
of $2\pi i$, so its real part vanishes and $h(z)=\Real\int_1^zq$ is
single-valued and harmonic on
$\C\setminus T$, continuous up to both sides of $T$, and $h(1)=0$.

On $(0,1)\setminus\{b_j\}$ we have $A<0 \le P^2$, so $q$ is purely
imaginary and $h$ is constant along each side of every subinterval. By
conjugation symmetry $h(\bar z)=h(z)$. Trace the upper side of $T$ from $1$
towards $0$, going up one side of each $\gamma_j$ and down the other. The
total variation of $h$ along this path is at most $2\sum_j\varepsilon_j$, so
$\abs h \le 2\sum_j\varepsilon_j$ at every boundary point of $\C\setminus T$.

For real $x>1$ we have $q(x)>0$ and
$\int_1^x(t(t-1))^{-1/2}dt=\log x+\log4+o(1)$, so
\[
 h(x)=\log x+\log4+\int_1^\infty\Bigl(q(t)-\frac1{\sqrt{t(t-1)}}\Bigr)dt+o(1).
\]
The substitution $t=1/(1-u^2)$ turns the last integral into $J$, using
$p/\sqrt a-1=sN/(\sqrt a(p+\sqrt a))$. Since $h-\log\abs z$ is harmonic at
infinity, $h(z)=\log\abs z+\log4+J+o(1)$ in every direction. The function
$h-g_T$, where $g_T(z)=\log\abs z-\log\capa T+o(1)$ is the Green function,
is bounded and harmonic on $\Chat\setminus T$, takes the value
$\log4+J+\log\capa T$ at infinity, and has boundary values at most
$2\sum_j\varepsilon_j$. The maximum principle gives the claim.
\end{proof}

When the arcs are given by polynomial parametrizations, the variation
integrals can be desingularized. Write the arc as $z(\tau)$,
$0 \le \tau \le 1$, with $z(0)=b_j$ and $z(1)=a_j$, and put $\tau=1-\sigma^2$.
Define the polynomials $L_j,M_j$ by $z(\tau)-a_j=(\tau-1)L_j(\tau)$ and
$z(\tau)-b_j=\tau M_j(\tau)$, and let $\widetilde P_j$, $\widetilde A_j$
be $P$ and $\prod_k(z-a_k)(z-\bar a_k)$ with the $j$-th factors removed.
Direct cancellation gives
\begin{equation}\label{up:eq:V}
 \Bigl(q\frac{dz}{d\sigma}\Bigr)^2=\tau^2V_j(\sigma),\qquad
 -V_j=\frac{4M_j^2\widetilde P_j(z)^2\,z'(\tau)^2}
 {z(z-1)L_j(\tau)(z-\bar a_j)\widetilde A_j(z)},
\end{equation}
where $z=z(\tau)$. For the curves used below, $z(\tau)=X_j(\tau)+iy_j\tau$
with $y_j>0$, we have $z'(\tau)\ne0$, and $V_j$ is analytic and nonzero at both
ends of the arc. Whenever $\Real(-V_j)>0$ we have
$\abs{\Real(q\,dz/d\sigma)}=\tau\abs{\Imag\sqrt{-V_j}}$ for either square
root, so
\begin{equation}\label{up:eq:eps}
 \int_{\gamma_j}\abs{\Real(q\,dz)}=\int_0^1\tau\,\abs[\big]{\Imag\sqrt{-V_j(\sigma)}}\,d\sigma .
\end{equation}

\subsection{The explicit domain}
Let $R=4066667/665901=6.10701\ldots$. For $j=1,\dots,5$, let
\[
 z_j(\tau)=X_j(\tau)+iy_j\tau,\qquad0 \le \tau \le 1,
\]
where $X_j$ is a polynomial of degree $28$ with exact rational Bernstein
coefficients and $y_j>0$ is rational. The exact data are in the
supplementary archive \cite{WikArchive}. The base is $b_j=X_j(0)$ and the
tip is $a_j=X_j(1)+iy_j$. Table~\ref{up:tab:arcs} lists the rounded values. The
circle $\abs w=R$ corresponds to $[m_-,m_+]=[0.880965\ldots,0.896573\ldots]$,
so two arcs end on the first slit, one on the second slit, and two on the
circle. Let $T$ be the resulting continuum and $D=D_T$ as in
Lemma~\ref{up:lem:quotient}. Thus $D$ is the disk $\abs w<R$ with the six
Goodman slits and thirty further analytic arcs removed.
Figure~\ref{up:fig:domain} shows $D$, drawn from the certified data.

\begin{figure}[ht]
\centering
\includegraphics[width=\textwidth]{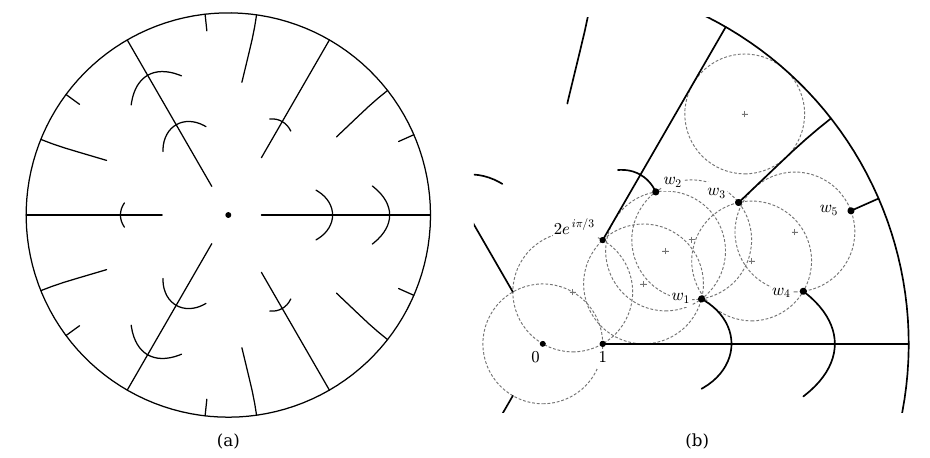}
\caption{(a) The domain $D$ of Theorem~\ref{thm:upper}: the disk
$\abs w<R$, $R=6.10701\ldots$, minus the six Goodman slits and thirty
further arcs. (b) The sector $0 \le \arg w \le \pi/3$ enlarged, with the tips
$w_1,\dots,w_5$ and eight numerically located, locally maximal disks of radius
approximately $1$ (dashed). Their centres are approximately $0$, $e^{i\pi/3}$, $1.6813+i$, $2.0495+1.5499i$,
$2.4852+1.7363i$, $3.4857+1.3836i$, $4.2077+1.8676i$ and $3.3701+3.8372i$.
The disk at $1.6813+i$ is tangent to $[1,R]$ and passes through
$2e^{i\pi/3}$.}
\label{up:fig:domain}
\end{figure}

\begin{table}[ht]
\caption{The five arcs, rounded: bases $b_j$ and tips $a_j=\Phi_R(w_j^3)$
in the quotient plane, the tips $w_j$ in the sector, and the point where
the lifted arc meets $\Gamma_R$ or the circle. Exact values are in
the archive \cite{WikArchive}.}
\label{up:tab:arcs}
\centering
\small
\begin{tabular}{cllll}
\toprule
$j$&$b_j$&$a_j$&$w_j$&$w$-plane base\\
\midrule
1&$0.8597948415$&$0.8597152018+0.0310768232i$&$2.6495+0.7498i$&$3.1499\in[1,R]$\\
2&$0.9150378101$&$0.9157589670+0.0093883090i$&$1.8853+2.5363i$&$3.2590\,e^{i\pi/3}$\\
3&$0.8919556962$&$0.8931896978+0.0118790690i$&$3.2673+2.3594i$&circle\\
4&$0.8791102560$&$0.8791140224+0.0048863808i$&$4.3482+0.8775i$&$4.8742\in[1,R]$\\
5&$0.8861278414$&$0.8860162233+0.0019199143i$&$5.1434+2.2204i$&circle\\
\bottomrule
\end{tabular}
\end{table}

\begin{theorem}\label{thm:upper}
For this domain, $\rho(D) \le 1+10^{-8}$ and
\[
 \frac{\rho(D)}{r(D,0)}<0.656268014.
\]
Consequently $\U<0.6562681$; this is Theorem~\ref{thm:upper-intro}.
\end{theorem}

The proof is a finite verification of four statements, described next.
All acceptance tests use Arb ball arithmetic \cite{Arb} at $192$-bit
precision, through \texttt{python-flint}~0.9.0. Floating-point computation
is used only to propose a witness point for a box, never to accept one.

\smallskip\noindent\emph{Exact checks.} The bases are distinct and lie in $(0,1)$, and $y_j>0$. Hence each arc meets
$\mathbb R$ only at $b_j$ and is a graph over the imaginary axis. For each
pair of arcs, the difference of their real parts at equal height is a
rational polynomial, and an exact Sturm count shows that it has no zero on
the common range. So the arcs are pairwise disjoint and
Lemma~\ref{up:lem:quotient} applies.

\smallskip\noindent\emph{Capacity.} The integrands of $J$ and of \eqref{up:eq:eps} are built from exact rational
polynomials in $u$ and in $\sigma$. The interval $[0,1]$ is divided into
$512$ subintervals for $J$ and $256$ for each $\varepsilon_j$. On a
subinterval with dyadic center $m$ and half-width $h$, every polynomial is
first shifted exactly to $m$. A power series evaluation at $m$ gives the
Taylor coefficients of order $<K$, with $K=10$ for $J$ and $K=8$ for
$\varepsilon_j$. The same evaluation on the ball $m\pm h$ encloses the
coefficient of order $K$ on the whole subinterval, which bounds the
integral-form remainder. Shifting first avoids the dependency blow-up of
evaluating expanded coefficients on a ball. On every ball the verifier also
checks $\Real(-V_j)>0$, so the principal square root in \eqref{up:eq:eps} is
analytic there. The integrand $\psi=\tau\Imag\sqrt{-V_j}$ of $\varepsilon_j$ may
change sign, so its absolute value is never expanded. Instead, if $g_k$ are the
Taylor coefficients of $\tau\sqrt{-V_j}$ at $m$ and $M$ bounds the coefficient
of order $K$ on the ball, the verifier uses
\[
 \int_{-h}^{h}\abs{\psi(m+v)}\,dv\le\sum_{k<K}\abs{\Imag g_k}\frac{2h^{k+1}}{k+1}
 +\frac{2Mh^{K+1}}{K+1}.
\] The results are
\[
 -0.0047990424887803<J<-0.0047990424887802,\qquad
 2\sum_{j=1}^5\varepsilon_j<1.1\cdot10^{-9},
\]
with $\varepsilon_1<5.1\cdot10^{-10}$, $\varepsilon_3<1.9\cdot10^{-11}$, and
the other three below $10^{-12}$.

\smallskip\noindent\emph{Inradius.} If $B(c,r)\subset D$ with $r>\rho_0=1+10^{-8}$, then $\abs c<R-\rho_0$, and
by symmetry we may take $0 \le \arg c \le \pi/3$. The verifier covers this region
by dyadic boxes. A box is accepted in one of three ways: it lies outside the
region; all of its corners are within $\rho_0$ of one Goodman slit; or all
of its corners are within $\rho_0$ of one boundary point. Distance to a
point or to a segment is convex, so checking the corners suffices. The
boundary points are Arb enclosures of $\Phi_R^{-1}(z_j(k/16384))$ for
$1 \le k \le 16384$, lifted to the sector and reflected in the lines
$\arg w=0$ and $\arg w=\pi/3$. The inverse uses only a square root and the
principal cube root of $-it$, and the correct root is the one in the upper
half-plane. A box that is not accepted is bisected. Exhausting the stack
proves $\rho(D) \le \rho_0$. Several extremal disks touch an arc at an
interior point. There a sample spacing $\delta$ costs about $\delta^2/8$,
and $4096$ samples per arc were not enough to close the disk centred near
$3.370+3.837i$.

\smallskip\noindent\emph{Conclusion.} By \eqref{up:eq:radius} and Lemma~\ref{up:lem:cap},
\[
 \Bigl(\frac{\rho(D)}{r(D,0)}\Bigr)^3 \le \rho_0^3\,\frac{\Delta_R}{4R^6}\,
 e^{-J+2\sum_j\varepsilon_j},
\]
and evaluating the right side in ball arithmetic gives
$\rho(D)/r(D,0) \le 0.6562680136$. Now \eqref{up:eq:reform} gives
Theorem~\ref{thm:upper}.

The verification completed successfully in about $76$ seconds on one
core. The inradius check visited $2243$ boxes, $1122$ of them terminal,
with maximal depth $58$. The verifier, the exact input data and the run
record, with hashes of the input and of the verifier, are in the archive
\cite{WikArchive}. The same verifier also gives $\U<0.6562745$ for a
simpler domain with three arcs per sector and $R=1843326/377609$, and it
confirms $\rho/r<0.65630373$ for the best two-arc configuration found in
Section~\ref{up:sec:search}.

\subsection{How the arcs were found}\label{up:sec:search}
\emph{This subsection describes the floating-point search. Nothing in it is
used in the proof of Theorem~\ref{thm:upper}.}

For prescribed tips $w_j$ in the sector, put $a_j=\Phi_R(w_j^3)$ and
require the arcs to be critical vertical trajectories of $q^2dz^2$ landing on
$[0,1]$. The conditions $\Real\int_{\Real a_j}^{a_j}q=0$ form an $n\times n$
linear system for the coefficients of $P$. We then maximize $\log r(D,0)$
over $R$ and the tips, subject to the constraint that every local maximum
of $\dist(\cdot,\partial D)$ is at most $1$. The same construction
reproduces the published values $0.6563936131$ \cite{CarrollOrtegaCerda}
and $0.6563192773$ \cite{OrtegaCerdaPridhnani} to about $10^{-12}$.
Table~\ref{up:tab:family} lists the best ratios found with $n$ arcs per
sector. The optimizer and the further experiments mentioned below are
described in the archive \cite{WikArchive}.

\begin{table}[ht]
\caption{Best ratios $\rho(D)/r(D,0)$ found numerically with $n$ arcs per
sector. The value for $n=5$ is certified in Theorem~\ref{thm:upper}.}
\label{up:tab:family}
\centering
\begin{tabular}{clll}
\toprule
$n$&$R$&ratio&remark\\
\midrule
$1$&$4.0546$&$0.6563936$&Carroll--Ortega-Cerd\`a \cite{CarrollOrtegaCerda}\\
$2$&$5.1195$&$0.6563193$&Ortega-Cerd\`a--Pridhnani \cite{OrtegaCerdaPridhnani}\\
$2$&$4.5956$&$0.6563037$&both tips free\\
$3$&$4.8816$&$0.6562744$&\\
$4$&$5.9036$&$0.6562682$&\\
$5$&$6.1070$&$0.6562680$&\\
\bottomrule
\end{tabular}
\end{table}

The gains from additional arcs have become small: going from three to four
arcs gains $6\cdot10^{-6}$, and going from four to five only
$2\cdot10^{-7}$. The following proposition shows that every Goodman-type
domain is limited by its first two generations.

\begin{proposition}\label{up:prop:core}
If $\Omega\ni0$ is simply connected and $\C\setminus\Omega$ contains the six
points $\omega^k$ and $-2\omega^k$, then $r(\Omega,0) \le (32/9)^{1/3}$. If
moreover $\D\subset\Omega$, then $\rho(\Omega)/r(\Omega,0) \ge (9/32)^{1/3}
=0.6551853\ldots$.
\end{proposition}

\begin{proof}
Put $Q=\{0,\omega^k,-\omega^k/2\}$, the inverses of the six points together with
$0$, and let $\kappa$ be the minimal capacity of a continuum containing $Q$. The
minimum is attained by a unique continuum $F$, a finite tree of analytic arcs
(Gr\"otzsch, Lavrentiev; see \cite{Dubinin}). Under $\zeta=1/z$ the domain
$\Omega$ maps onto the unbounded component of the complement of the continuum
$\{1/z:z\in\Chat\setminus\Omega\}$, which contains $Q$. Hence $r(\Omega,0) \le
1/\kappa$. By uniqueness, $F$ is invariant under conjugation and under rotation
by $\omega$. Since $F$ is a tree, it contains a unique arc joining $-1/2$ and
$1$. Conjugation maps this arc to itself and fixes its endpoints, so it fixes
the arc pointwise. Hence $[-1/2,1]\subset F$. By rotation, $F$ contains the
six-armed star~$X$, and~$X$ itself contains $Q$. Hence $\kappa=\capa X$. Since
$X$ is the full preimage of $[-1/8,1]$ under $z\mapsto z^3$, we get $\capa
X=(9/32)^{1/3}$.
\end{proof}

Theorem~\ref{thm:upper} is within $1.1\cdot10^{-3}$ of this floor. In our
searches, moving the starting points $1$ and $2$ of the Goodman slits made
the ratio larger.

The proposition does not cover one modification. The second-generation tip
need not lie on the bisector: any point $e^{i\pi/3}+e^{i\theta}$ with
$0<\theta<2\pi/3$ blocks the disk centered at $e^{i\pi/3}$. Numerically,
the resulting rotationally symmetric cores have P\'olya--Chebotarev values
below $(9/32)^{1/3}$, for example about $0.6534$ at $\theta=\pi/6$. In
their extremal continua each first-generation slit merges with an adjacent
off-axis slit, which leaves a wide sector open, and in our attempts filling
these sectors cost more than the core gains. We have not searched such
chiral configurations systematically, so whether one of them beats
Theorem~\ref{thm:upper} remains open.

\section{Further directions}\label{sec:status}

For the lower bound, Proposition~\ref{prop:centered} rules out disks
centered at $f(0)$ beyond $0.6$. It remains open whether
$\inf_{f\in\B}d_f\ge0.592$, that is, whether Theorem~\ref{thm:hand} could
also be obtained from a centered disk. Additional Grunsky coefficients, further
consequences of Bloch normalization, or a different center construction may
strengthen that argument.

Floating-point estimates suggest that the method of Section~\ref{sec:hybrid}
has little further room. At $s=1/3$, $a=0$ the twelve regions allow a radius of
about $0.627$. Near the edge $s=1/3$, $\abs a\approx b(s)$, a single center
gives only about $0.620$, both for the rows $\{1,3,5\}$ and for
$\{1,3,5,7\}$, so there the point bounds of Lemma~\ref{hyb:lem:bloch} are
essential. Even with a center depending on all the coefficients involved,
these relaxations appear to stay below $0.635$.

For the upper bound, Proposition~\ref{up:prop:core} limits domains built
on Goodman's first two generations, so the off-axis second-generation tips
discussed at the end of Section~\ref{sec:upper} are the natural candidates
for a substantially different domain. Neither approach presently establishes that
the known upper-bound domains are close to a global extremum.

\appendix
\section{Finite scalar estimates for the analytic bound}\label{app:scalar}

The proof of Theorem~\ref{thm:hand} uses \eqref{eq:smallchecks}, the six
table entries, and the three bounds $H_\lambda>22$. The two comparisons in
\eqref{eq:smallchecks} hold with margins above $3\cdot10^{-4}$. The table
entries hold with margins above $10^{-4}$; the smallest is
$E_0(1/3)-0.058\approx1.5\cdot10^{-4}$ for $\lambda=4$, while $E_0$ and $E_1$
themselves exceed $0.011$. The bounds for $H_\lambda$ hold with margins above
$0.1$. So five-digit decimals suffice. For completeness
the transcendental inputs can be enclosed with rational arithmetic. For
$0<v<1$ put
\[
 L(v)=2\sum_{j=0}^{7}\frac{v^{2j+1}}{2j+1},\qquad
 \epsilon(v)=\frac{2v^{17}}{17(1-v^2)},
\]
so that $L(v) \le \log((1+v)/(1-v)) \le L(v)+\epsilon(v)$. Writing
$2\artanh t=\log(1997/3)=9\log2+\log(1997/1536)$ and using
$v=1/3$ and $v=461/3533$ gives
\[
 0.0940845<q<0.0940846.
\]
The alternating Taylor sums of $e^{-5/8}$ through degrees $11$ and $10$ are
lower and upper bounds. Substitution gives
\begin{align*}
 \tfrac{13}{4}e^{-5/8}&<1.7396<1.7399<A_R,\\
 \tfrac92e^{-5/4}-1&>0.2892>0.2890>q+\tfrac{196}{1000}\,t^2,
\end{align*}
which is \eqref{eq:smallchecks}. The same substitution gives the table in
the proof and $H_6>41$, $H_5>31$, $H_4>22$. The bounds $q<t^2$,
$R+1/11<7/10$ and $H_\lambda>22$ give the geometric cutoffs used in the
proof. All decimals in these comparisons denote exact terminating
rationals.

The script \path{check-hand-bound.py} in the archive \cite{WikArchive}
performs these rational substitutions with Python's standard
\texttt{fractions} module. It
only checks the arithmetic: the proof consists of the displayed inequalities
and involves no adaptive subdivision. The archive also contains a Lean~4
formalization. It verifies these estimates from Mathlib's bounds for the
exponential and the logarithm, and proves Theorem~\ref{thm:hand} with the
Grunsky inequality \eqref{eq:M} and the Fekete--Szeg\H{o} inequality as
explicit hypotheses.

\section{Proof of Proposition~\ref{prop:centered}}\label{app:centered}

Expansion at the origin gives
$f_0(z)=z-\frac{\eta^2}{3} z^3 - \frac{a\eta^3}{2} z^4+\cdots$. Put
$w(z)=-z^3(a+z)/(1+az)$. Since $(a+z)/(1+az)$ is an automorphism of $\D$,
$\abs{w(z)} \le \abs{z}^3$, and a direct computation gives
\[
 H(z)=\frac{1+w(z)}{\bigl(1-w(z)\bigr)(1+z^2)} .
\]
The function $g(z)=\eta^{-1}\arctan(\eta z)$ is convex in $\abs{z}<1/\eta$,
since $\Real(1+zg''/g')=\Real\bigl((1-\eta^2z^2)/(1+\eta^2z^2)\bigr)>0$ there,
and $f_0'/g'=(1+w(\eta z))/(1-w(\eta z))$ has positive real part. So $f_0\circ
g^{-1}$ has a derivative with positive real part on a convex domain and is
therefore injective. Thus $f_0$ is univalent in $\abs{z}<1/\eta$; in particular
$f_0(1)\notin f_0(\D)$.

Write $H=N/D$. The inequality $(1-\abs{z}^2)\abs{f_0'(z)} \le 1$ on $\D$ is
equivalent to
\[
 \abs{D(\eta z)}^2-(1-\abs{z}^2)^2\abs{N(\eta z)}^2 \ge 0 .
\]
The left side is even in $\Imag z$. With $r=\abs{z}$ and $\Real z=r(2\tau-1)$ it
becomes $r^2$ times a polynomial of bidegree $(10,6)$ in $(r,\tau)\in[0,1]^2$.
Subdividing the square into $65$ rectangles, all Bernstein coefficients of this
polynomial are positive. This is checked in exact rational arithmetic
\cite{WikArchive}. Hence the Bloch inequality holds, strictly for $z\ne0$, and
$f_0\in\B$.

Finally, the numerator of $H'(z)$ is
$2z(25z^8+20z^7+4z^6-30z^5-108z^4-60z^3-112z^2-50z-25)$, which is negative for
$0<z \le 1$, because the three positive terms are at most $49z^2$. So $H$
decreases on $[0,1]$, and a left Riemann sum with $16384$ intervals, evaluated
in exact arithmetic with upward rounding, gives
$f_0(1)=\int_0^1H(\eta t)\,dt<0.599847$.\qed

\begingroup
\raggedright

\endgroup

\end{document}